\documentclass[12pt]{amsart}
\date{}

\newtheorem{theorem}{Theorem}
\newtheorem{lemma}{Lemma}
\newtheorem{remark}{Remark}
\newtheorem{corollary}{Corollary}
\newtheorem{proposition}{Proposition}

\newtheorem{example}{Example}

\newcommand{\e}{\varepsilon}

\newcommand{\strilka}{\kern0pt\uparrow\kern-3pt}

\newcommand{\IN}{{\mathbb N}}
\newcommand{\IR}{{\mathbb{R}}}
\newcommand{\w}{\omega}
\newcommand{\IT}{{\mathbb{T}}}

\newcommand{\IZ}{{\mathbb{Z}}}

\begin{document}
\title[On subgroups of second countable divisible abelian groups]
{Each second countable abelian group is a subgroup of a second
countable divisible group}
\author{Taras Banakh}
\email{tbanakh@franko.lviv.ua}
\author{Liubomyr Zdomsky\u\i}
\address{Department of Mathematics, Ivan Franko Lviv National University,
Universytetska 1, Lviv, 79000, Ukraina}
\begin{abstract}
It is shown that each pseudonorm $|\cdot|_H$ defined on a subgroup
$H$ of an abelian group $G$ can be extended to a pseudonorm
$|\cdot|_G$ on $G$ such that the densities of the pseudometrizable
topological groups $(H,|\cdot|_H)$ and $(G,|\cdot|_G)$ coincide.
We derive from this that any Hausdorff $\omega$-bounded group
topology on $H$ can be extended to a Hausdorff $\omega$-bounded
group topology on $G$. In its turn this result implies that each
separable metrizable abelian group $H$ is a subgroup of a
separable metrizable divisible group $G$. This result essentially
relies on the Axiom of Choice and is not true under the Axiom of
Determinacy (which contradicts to the Axiom of Choice but implies
the Countable Axiom of Choice).
\end{abstract}

\maketitle

This paper was motivated by the following question having its
origin in functional analysis (see \cite{PZ}, \cite{BPZ}): {\em Is
it true that every metrizable separable abelian topological group
with no torsion is a subgroup of a metrizable separable divisible
abelian group with no torsion?}

{\em From now on all groups considered in the paper are
commutative}. We recall that a group $G$ is {\em divisible} (resp.
{\em has no torsion}) if for any element $a\in G$ and a positive
integer $n$ the equation $nx=a$ has a solution $x\in G$ (resp.
does not have two distinct solutions in $G$). According to the
Baer Theorem \cite[21.1]{F} each divisible group $G$ is {\em
injective} in the sense that each homomorphism $h:B\to G$ defined
on a subgroup $B$ of a group $A$ can be extended to a homomorphism
$\bar h:A\to G$. A classical result of the theory of infinite
abelian groups \cite[24.1]{F} asserts that each group (with no
torsion) is a subgroup of a divisible group (with no torsion).
This result allows us to reduce the above question to the
following one: {\em Can every separable group topology on a
subgroup $H$ of a group $G$ be extended to a separable group
topology on $G$?}

Note that without the separability requirement this problem is
trivial: just announce $H$ to be an open subgroup of $G$ and take
the neighborhood base at the origin of $H$ for a neighborhood base
at the origin in the group $G$. However if the quotient group
$G/H$ is uncountable such an extension leads to an unseparable
topology on $G$. So, another less direct approach should be
developed.

A classical result in the theory of topological groups asserts
that each group topology is generated by a family of continuous
pseudonorms, see \cite[\S2]{Tk}. This observation allows us to
reduce the problem of extending group topologies to the problem of
extending pseudonorms. As usual, under a {\em (continuous)
pseudonorm} of a (topological) group $G$ we understand a
(continuous) non-negative function $|\cdot|:G\to [0,\infty)$ such
that $|0|=0$ and $|x-y|\le |x|+|y|$ for all $x,y\in G$. A
pseudonorm $|\cdot|$ is a {\em norm} provided $|x|=0$ implies
$x=0$. Each pseudonorm $|\cdot|$ on a group $G$ generates a group
topology on $G$ whose neighborhood base at the origin consists of
the $\e$-balls $B_{|\cdot|}(\e)=\{x\in G:|x|<\e\}$, $\e>0$. The
group $G$ endowed with this topology turns into a topological
group denoted by $(G,|\cdot|)$.

Given a topological space $X$ by $d(X)$ we denote its {\em
density} (that is the smallest size of a dense subset of $X$), by
$w(X)$ its {\em weight} (that is the smallest size of a base of
the topology of $X$) and by $\chi(X)$ its {\em character} (i.e., a
smallest cardinal $\tau$ such that any point $x\in X$ possesses a
neighborhood base $\mathcal B$ of size $|\mathcal B|\le\tau$). It
is known that $d(X)=w(X)$ for any (pseudo)metrizable topological
space.

Now we are able to formulate the main result of this paper.

\begin{theorem}\label{main} Any pseudonorm $|\cdot|_H$ defined on a
subgroup $H$ of an abelian group $G$ can be extended to a
pseudonorm $|\cdot|_G$ on $G$ so that
$d(H,|\cdot|_H)=d(G,|\cdot|_G)$.
\end{theorem}

Because of its technical character we postpone the proof of this
theorem till the end of the paper. Now we consider some its
corollaries.

According to \cite[4.1]{Tk}, $w(G)=\chi(G)\cdot ib(G)$ for any
topological group $G$ where $ib(G)$ stands for the {\em
boundedness index} of $G$, equal to the smallest cardinal $\tau$
such that for any neighborhood $U$ of the origin of $G$ there is a
subset $F\subset G$ with $G=F\cdot U$ and $|F|\le\tau$, see
\cite[\S3]{Tk}. Topological groups $G$ with $ib(G)\le\aleph_0$ are
called {\em $\omega$-bounded}, see \cite{Gu} or \cite{Tk}. It is
known that a metrizable topological group is $\omega$-bounded if
and only if it is separable. Unlike to separable groups, the class
of $\omega$-bounded groups is closed under many operations, in
particular taking subgroups and Tychonov products, see \cite{Tk}
or \cite{Gu}.

Taking into account that $|X|\le 2^{w(X)}$ for any Hausdorff
topological space $X$ \cite[1.5.1]{En} and $w(G)=\chi(G)\cdot
ib(G)$ for any topological group $G$, we get $|G|\le 2^{\chi(G)}$
for any Hausdorff $\omega$-bounded topological group $G$. This
inequality can be rewritten as $\chi(G)\ge \log |G|$, where $\log
\kappa=\min\{\tau: k\le 2^\tau\}$ for a cardinal $\kappa$.
\smallskip

\begin{theorem}~\label{exttop} Any Hausdorff group topology $\tau_H$ defined on
a subgroup $H$ of an abelian group $G$ can be extended to a
Hausdorff group topology $\tau_G$ on $G$ so that
$ib(G,\tau_G)=ib(H,\tau_H)$,
$\chi(G,\tau_G)=\max\{\chi(H,\tau_H),\log |G|\}$ and
$w(G,\tau_G)=\max\{w(H,\tau_H),\log |G|\}$.
\end{theorem}

In an obvious way Theorem~\ref{exttop} implies

\begin{corollary}\label{extsep} Any separable metrizable topology defined on a
subgroup $H$ of an abelian group $G$ with $|G|\le\mathfrak c$ can
be extended to a separable metrizable topology on $G$.
\end{corollary}

Here $\mathfrak c$ stands for the size of continuum. The next our
corollary follows from Theorem~\ref{exttop} and Theorem  24.1 of
\cite{F} asserting that each abelian group $H$ (with no torsion)
is a subgroup of a divisible group $G$ (with no torsion) such that
$|G|=|H|$.

\begin{corollary} Any Hausdorff topological abelian group $H$
(with no torsion) is a subgroup of a Hausdorff abelian divisible
group $G$ (with no torsion) such that $w(G)=w(H)$,
$\chi(G)=\chi(H)$ and $ib(G)=ib(H)$.
\end{corollary}

The following particular case of the above corollary gives a
positive answer to the question stated at the beginning of the
paper.

\begin{corollary}\label{separ} Each separable metrizable abelian group $H$ (with no
torsion) is a subgroup of a separable metrizable divisible group
$G$ (having no torsion).
\end{corollary}

In fact, the construction of such a divisible group $G\supset H$
hardly uses Axiom of Choice (see Remark~\ref{rem}). As a result
the group $G$ has a complex descriptive structure. We shall show
that in general the group $G$ is not analytic. Let us recall that
a topological space is {\em analytic} if it is a metrizable
continuous image of a Polish space. As usual, under a {\em Polish
space} we understand a topological space homeomorphic to a
separable complete metric space. A topological group is {\em
Polish} ({\em analytic}) if its underlying topological space is
Polish (analytic).

The well-known Open Mapping Principle for Banach spaces
generalizes to topological groups as follows: {\em Any continuous
group homomorphism from an analytic group onto a Polish group is
open}. The proof of this Open Mapping Principle follows from
Theorem 9.10 \cite{Ke} asserting that any homomorphism $h:G\to H$
from a Polish group $G$ into a $\omega$-bounded group $H$ is
continuous provided $h$ has the Baire Property and Theorem 29.5 of
\cite{Ke} asserting that analytic subspaces of Polish spaces have
Baire Property. We remind that a subset $A$ of a topological space
$X$ has {\em the Baire property} if $A$ contains a
$G_\delta$-subset $G$ of $X$ such that $A\setminus G$ is meager in
$X$.

For a group $H$ with no torsion and a positive integer $n$ let
$nH=\{ny:y\in H\}\subset H$ and $\frac1n:nH\to H$ be the map
assigning to each element $x\in nH$ a unique $y\in H$ such that
$ny=x$.

\begin{proposition}\label{anal} If a Polish group $H$ is a subgroup of a
divisible analytic group $G$ with no torsion, then for every
positive integer $n$ the map $\frac1n:nH\to H$ is continuous.
\end{proposition}

\begin{proof} The subgroup $H$, being complete, is closed in $G$.
Then the subgroup $\frac1nH=\{g\in G:ng\in H\}$, being the
preimage of $H$ under the continuous map $n:G\to G$, $n:x\mapsto
nx$, is a closed subset of $G$ and thus is analytic. Since the
group $G$ is divisible and has no torsion, the map $n:\frac1nH\to
H$, $n:x\to nx$, is a bijective continuous group homomorphism from
the analytic group  $\frac1nH$ onto the Polish group $H$. Applying
the Open Mapping Principle for topological groups we conclude that
this map is a topological isomorphism and hence the map
$\frac1n:H\to\frac1nH$ is continuous. Since $nH\subset H$, the map
$\frac1n:nH\to H$ is continuous too.
\end{proof}

Finally we give an example of a Polish group without torsion
admitting no embedding into a divisible analytic group without
torsion.

\begin{example}\label{ex} There is a Polish group $H$ without torsion such
that the map $\frac12:2H\to H$ is discontinuous. This group $H$
cannot be a subgroup of a divisible analytic group with no
torsion.
\end{example}

\begin{proof} For every $k\in\IN$ let $H_k$ be a copy of the group
$\IR$ of reals and let $e_k=1\in H_k$. Endow the group $H_k$ with
the norm $|x|_k=\sqrt{(\cos(\pi z)-1)^2+\sin^{2}(\pi
z)+(2^{-(k+1)}x)^2}$ (which is generated by the usual Euclidean
distance under a suitable winding of $H_k=\IR$ around a cylinder
in $\IR^3$). It is easy to verify that $|e_k|_k>2$ while
$|2e_k|_k=2^{-k}$.

On the direct sum $\oplus_{k\in\IN}H_k$ consider the norm
$|(x_k)_{k\in\IN}|= \sum_{i\in\IN} |x_k|_k$ and let $H$ be the
completion of $\oplus_{k\in\IN}H_k$ with respect to this norm.
Then $H$ is a Polish group. We claim that $H$ has no torsion.

Consider the identity inclusion $i:\oplus_{k\in\IN}H_k\to
\prod_{k\in\IN}H_k$ from the direct sum into the direct product
endowed with the Tychonov topology. Observe that this direct
product is a complete group. To show that the group $H$ has no
torsion, it suffices to verify that the extension $\bar i:H\to
\prod_{k\in\IN}H_k$ of the homomorphism $i$ onto the completion
$H$ is injective.

It will be convenient to think of elements of the groups
$\oplus_{k\in\IN}H_k$ and $\prod_{k\in\IN}H_k$ as functions
$f:\IN\to\bigcup_{k\in\IN}H_k$.

Assuming that the homomorphism $\bar i$ is not injective, we could
find an element $f_\infty\in H$ such that $f_\infty\ne 0$ but
$\bar i(f_\infty)=0$. Fix any $\e>0$ with $\e<|f_\infty|$. Choose
a sequence $(f_n)_{n\in\IN}\in \oplus_{k\in\IN}H_k$ converging to
$f_\infty$ in $H$. We can assume that $|f_n|>\e$ for every
$n\in\IN$. By the continuity of the map $\bar i$, we conclude that
the sequence $\{i(f_n)\}_{n\in\IN}$ converges to zero in
$\prod_{k\in\IN}H_k$ (this means that the function sequence
$(f_n)$ is pointwise convergent to zero). Since the sequence
$(f_n)$ is Cauchy in $H$, there is $m\in\IN$ such that
$|f_m-f_j|<\frac{\e}2$ for any $j\ge m$. Without loss of
generality, we can assume that $f_m(k)=0$ for all $k> m$. Since
for every $k$ \ $\lim_{j\to\infty}f_j(k)=0$, we can find $j>m$ so
large that $|f_j(k)|_k<\frac{\e}{2m}$ for all $k\le m$. Then
$|f_m-f_j|=\sum_{k=1}^\infty |f_m(k)-f_j(k)|_k\ge
\sum_{k=1}^m|f_m(k)-f_j(k)|_k\ge
\sum_{k=1}^m|f_m(k)|_k-\sum_{k=1}^m|f_j(k)|_k\ge
|f_m|-\sum_{k=1}^m\frac{\e}{2m}> \e-\frac{\e}{2}=\frac{\e}2,$
which contradicts to $|f_m-f_j|<\frac{\e}2$.

Therefore, the homomorphism $\bar i$ is injective and the group
$H$ has no torsion. Since $|e_{k}|=|e_k|_k>2$ and
$|2e_{k}|=|2e_k|_k=2^{-k}$ for all $k$, we see that the sequence
$(2e_k)$ converges to zero in $H$ while $(e_k)$ does not. This
means that the map $\frac12:2H\to H$ is discontinuous.
\end{proof}.

\begin{remark}\label{rem}\rm Corollary~\ref{separ} can not be proven
without the full Axiom of Choice and is not true under the Axiom
of Determinacy. This axiom contradicts the Axiom of Choice but
implies its weaker form, the Countable Axiom of Choice, see
\cite[\S9.2 and \S9.3]{JW}. It is known that under the Axiom of
Determinacy, any subset of a Polish space has the Baire Property,
see \cite[8.35]{Ke}. This fact and Theorem 9.10 of \cite{Ke}
implies that under Axiom of Determinacy the Open Mapping Principle
for topological groups holds in the following more strong form:
any continuous homomorphism $h:H\to G$ from a $\omega$-bounded
group $H$ onto a Polish group $G$ is open. Using this stronger
form of the Open Mapping Principle and repeating the proof of
Proposition~\ref{anal} we see that under the Axiom of Determinacy
this proposition holds without the analycity assumption on the
group $G$. Thus we come to a rather unexpected conclusion: {\em
Under the Axiom of Determinacy the group $H$ from Example~\ref{ex}
cannot be embedded into a metrizable separable divisible group
with no torsion, in spite of the fact that algebraically, $H$ is a
subgroup of the countable product $\IR^\omega$ of lines}. This
shows that Corollary~\ref{separ} is not true under the Axiom of
Determinacy.
\end{remark}

\section{Proof of Theorem~\ref{main}}

In the proof of Theorem~\ref{main} we shall need one combinatorial
lemma. A collection $\mathcal A$ of subsets of a set $X$ is called
{\em $k$-uniform} where $k\in\IN$ if $|A|=k$ for each
$A\in{\mathcal A}$; ${\mathcal A}$ is {\em disjoint} if it
consists of pairwise disjoint sets.

\begin{lemma}\label{combi} Suppose $k\in\IN$ and $\mathcal A$, $\mathcal B$ are two
disjoint $k$-uniform finite collections of subsets of an infinite
set $X$. Then there is a subset $I\subset X$ such that $|I\cap
C|=1$ for each $C\in\mathcal A\cup\mathcal B$.
\end{lemma}

\begin{proof} It is easy to construct $k$-uniform disjoint finite
collections ${\mathcal C}$, ${\mathcal D}$ of subsets of $X$ such
that ${\mathcal A}\subset {\mathcal C}$, ${\mathcal
B}\subset{\mathcal D}$, $|{\mathcal C}|=|{\mathcal D}|$, and
$\cup{\mathcal C}=\cup{\mathcal D}$. Let $n=|{\mathcal
C}|=|{\mathcal D}|$ and write ${\mathcal C}=\{C_1,\dots,C_n\}$,
${\mathcal D}=\{D_1,\dots,D_n\}$. Consider the matrix
$[a_{ij}]_{i,j=1}^n$ where $a_{ij}=\frac1k|C_i\cap D_j|$ and
observe that it is {\em double stochastic}, that is
$\sum_{i=1}^na_{ij}=1=\sum_{j=1}^na_{ij}$ for all
$i,j\in\{1,\dots,n\}$. According to the Birkhoff Theorem (see
\cite{Bi}, \cite[p.556]{Ga}, or \cite[8.40]{A}) each double
stochastic matrix is a convex combination of {\em permutating
matrices}, that is matrices of the form
$[\delta_{i,\sigma(j)}]_{i,j=1}^n$ where $\sigma$ is a permutation
of the set $\{1,\dots,n\}$ and $[\delta_{ij}]_{i,j=1}^n$ is the
identity matrix. This result implies the existence of a
permutation $\sigma$ of the set $\{1,\dots,n\}$ such that
$a_{i,\sigma(i)}>0$ for all $i$. This means that the intersection
$C_i\cap D_{\sigma(i)}$ is not empty and thus contains some point
$x_i$. Let $I=\{x_1,\dots,x_n\}$ and observe that $|C\cap I|=1$
for any element $C\in{\mathcal C}\cup{\mathcal D}\supset{\mathcal
A}\cup{\mathcal B}$.
\end{proof}

Theorem~\ref{main} will be proved by induction whose inductive
step is based of the following

\begin{lemma} Let $H$ be a subgroup of a group $G$ such that
$pG\subset H$ for some prime number $p$. Then any pseudonorm
$|\cdot|_H$ on $H$ can be extended to a pseudonorm $|\cdot|_G$ on
$G$ so that $d(G,|\cdot|_G)=d(H,|\cdot|_H)$.
\end{lemma}

\begin{proof} The quotient group $G/H$ has prime exponent $p$ and thus
has a basis which can be written as $\{g^\alpha+H:\alpha<\mu\}$
for some ordinal $\mu$, see \cite[16.4]{F}. It will be convenient
to complete this basis by zero letting $g^\mu=0$. It follows that
any element of $G$ can be represented in the form
$x=u+\sum_{i\in\w}g^{\alpha_i}$ where $u\in H$ and the set $\{i\in
\w:\alpha_i\ne\mu\}$ is finite.  For any element $x\in G$ let
$
\mathrm{Rep}(x)=\{(u,(\alpha_{i})_{i\in\w})\in H\times [0,\mu]^\w:
x=u+\sum_{i\in\w}g^{\alpha_{i}} \}$. Observe that for any $x\in H$
and $(u,(\alpha_i)_{i\in\w})\in\mathrm{Rep}(x)$ the number $p$
divides the cardinality of the set $\{i\in\w:\alpha_i=\alpha\}$
for each ordinal $\alpha<\mu$.

Let $|\cdot|_H$ be a pseudonorm on $H$. Define a function
$\rho:G\times G\to[0,\infty)$ letting
\[
\rho(x,y)=\inf\{|u-v|_H+\sum_{i\in\w}|pg^{\alpha_{i}}-pg^{\beta_{i}}|_H:
(u,(\alpha_{i})_{i\in\w})\in\mathrm{Rep}(x),(v,(\beta_{i})_{i\in\w})\in
\mathrm{Rep}(y)\}
\]
for $x,y\in G$.

It is easy to see that $\rho$ is an invariant pseudometric on $G$.
Let us show that $d(G,\rho)\le d(H,|\cdot|_H)$. Let $D$ be a dense
subset of $(H,|\cdot|_H)$ with $|D|=d(H,|\cdot|_H)$ and
$I\subset[0,\mu]$ be a subset of size $|I|\le d(H,|\cdot|_H)$ such
that $I\ni\mu$ and the set $\{pg^\alpha:\alpha\in I\}$ is dense in
the subspace $\{pg^\alpha:\alpha<\mu\}$ of $(H,|\cdot|_H)$. Then
the set $E=\{x\in G: \mathrm{Rep}(x)\cap ({D}\times
I^\w)\not=\emptyset\}$ is a dense subset of $(G,\rho)$ with
$|E|\le d(H,|\cdot|_H)$. This proves the inequality $d(G,\rho)\le
d(H,|\cdot|_H)$.

It remains to show that $\rho(x,y)=|x-y|_H$ for any $x,y\in H$.
Fix arbitrary $(u,(\alpha_{i})_{i\in\w})\in\mathrm{Rep}(x)$,
$(v,(\beta_{i})_{i\in\w})\in\mathrm{Rep}(y)$. For every
$\alpha<\mu$ let $A(\alpha)=\{i\in\w:\alpha_i=\alpha\}$ and
$B(\alpha)=\{i\in\omega:\beta_i=\alpha\}$. Since $x,y\in H$, the
number $p$ divides the cardinalities of the sets $A(\alpha)$,
$B(\alpha)$ for all ordinals $\alpha<\mu$. Applying
Lemma~\ref{combi}, find a subset $I\subset\w$ such that $|C\cap
I|=\frac1p|C|$ for every nonempty subset
$C\in\{A(\alpha),B(\alpha):\alpha<\mu\}$. Then
\begin{eqnarray*}
x=&u+\sum_{i\in\w}g^{\alpha_i}=&u+\sum_{\alpha<\mu}|A(\alpha)|\cdot
g^\alpha=u+\sum_{\alpha<\mu}p|I\cap A(\alpha)|\,g^\alpha=\\
&&=u+\sum_{\alpha<\mu}\sum_{i\in I\cap
A(\alpha)}pg^{\alpha_i}=u+\sum_{i\in I}pg^{\alpha_i}.
\end{eqnarray*}
By analogy, $y=v+\sum_{i\in I}pg^{\beta_i}$. Consequently,
\begin{eqnarray*}
|u-v|_H+&\sum_{i\in\w}|pg^{\alpha_{i}}-pg^{\beta_{i}}|_H\geq
|u-v|_H+\sum_{i\in I}|pg^{\alpha_{i}}-pg^{\beta_{i}}|_H\geq
\\
&\geq |(u+\sum_{i\in I} pg^{\alpha_{i}})-(v+\sum_{i\in I}
pg^{\beta_{i}})|_H=|x-y|_H
\end{eqnarray*}
Passing to the infimum, we get $\rho(x,y)\geq|x-y|_H$. The proof
of the inverse inequality is straightforward, hence
$\rho(x,y)=|x-y|_H$. Letting $|x|_G=\rho(x,0)$ for $x\in G$ we
define a pseudonorm on $G$ extending the pseudonorm $|\cdot|_H$ so
that $d(G,|\cdot|_G)=d(H,|\cdot|_H)$.
\end{proof}

\begin{lemma} Let $H$ be a subgroup of a group $G$ such that the
quotient group $G/H$ is periodic. Then any pseudonorm $|\cdot|_H$
on $H$ can be extended to a pseudonorm $|\cdot|_G$ on $G$ so that
$d(G,|\cdot|_G)=d(H,|\cdot|_H)$.
\end{lemma}

\begin{proof} Let
$(p_i)_{i=1}^\infty$ be a sequence of prime numbers such that for
every prime number $p$ the set $\{i\in\IN:p_i=p\}$ is infinite.
Let $H_0=H$ and for $i\ge 0$ let
$H_{i+1}=\frac1{p_{i+1}}H_i=\{x\in G:p_{i+1}x\in H_i\}$. Because
of the periodicity of the quotient group $G/H$ we get
$G=\bigcup_{i=1}^\infty H_i$.

Let $|\cdot|_H$ be any pseudonorm on $H$ and $D_0$ be a dense
subset of the topological group $(H,|\cdot|_H)$ with
$|D_0|=d(H,|\cdot|_H)$. Let $|\cdot|_0=|\cdot|_H$. Using the
previous lemma, by induction for every $i\ge 1$ find a pseudonorm
$|\cdot|_i$ on the group $H_i$ and a dense subset $D_i$ of the
topological group $(H_i,|\cdot|_i)$ such that $|x|_i=|x|_{i-1}$
for each $x\in H_{i-1}$ and $|D_i|=|D_{i-1}|$.

Completing the inductive construction, define a pseudonorm
$|\cdot|_G$ on the group $G$ letting $|x|_G=|x|_i$ where $x\in
H_i$. It is clear that $|\cdot|_G$ extends $|\cdot|_H$ and
$D=\bigcup_{i=1}^\infty D_i$ is a dense set in the topological
group $(G,|\cdot|_G)$ with $|D|=|D_0|=d(H,|\cdot|_H)$. This yields
$d(G,|\cdot|_G)\le d(H,|\cdot|_H)$.
\end{proof}

Finally we are able to complete the
\smallskip

\noindent {\em Proof of Theorem~\ref{main}.} Let $H$ be a subgroup
of a group $G$ and $|\cdot|_H$ be a pseudonorm on $H$. According
to \cite[24.1]{F} the group $H$ is a subgroup of a divisible group
$E$. Moreover, according to Lemma 24.3 \cite{F} we can assume that
the quotient group $E/H$ is periodic. Applying the previous lemma,
extend the pseudonorm $|\cdot|_H$ to a pseudonorm $|\cdot|_E$ on
$E$ so that $d(E,|\cdot|_E)=d(H,|\cdot|_H)$. According to Baer
Theorem \cite[21.1]{F}, each divisible group is injective.
Consequently, there is a group homomorphism $h:G\to E$ extending
the identity map $H\to H\subset E$. Define a pseudonorm
$|\cdot|_G$ on $G$ letting $|x|_G=|h(x)|_E$ for $x\in G$ and
observe that $|\cdot|_G$ extend $|\cdot|_H$ and $d(H,|\cdot|_H)\le
d(G,|\cdot|_G)\le d(E,|\cdot|_E)=d(H,|\cdot|_H)$. \qed \smallskip

\section{Proof of Theorem~\ref{exttop}}

Let $H$ be a subgroup of a group $G$ and $\tau_H$ be a Hausdorff
group topology on $H$.

First we define a Hausdorff group topology $\tau_{G/H}$ on the
quotient group $G/H$ such that $ib(G/H,\tau_{G/H})\le \omega$ and
$\chi(G/H,\tau_{G/H})\le\log |G/H|$. By \cite[24.1]{F} $G/H$ is a
subgroup of a divisible group $E$ with $|E|=|G/H|$. Applying
Theorem 23.1 \cite{F} (on the structure of divisible groups), we
can show that the group $E$ is isomorphic to a subgroup of the
power $\IT^\kappa$ of the circle $\IT=\IR/\IZ$ where $\kappa=\log
|G/H|\le\log |G|$. Observe that $\IT^\kappa$ endowed with the
natural Tychonov product topology is a compact topological group
with $ib(\IT^\kappa)=\omega$ and
$\chi(\IT^\kappa)=\kappa\le\log|G|$.

Consequently, the group $G/H$, being isomorphic to a subgroup of
$\IT^\kappa$, carries a Hausdorff group topology $\tau_{G/H}$ such
that $ib(G/H,\tau_{G/H})\le \omega$ and $\chi(G/H,\tau_{G/H})\le
\kappa\le \log |G|$.

Fix a neighborhood base $\mathcal B$ of size $|\mathcal
B|=\chi(H,\tau_H)$ at the origin of the topological group
$(H,\tau_H)$. Applying \cite[2.3]{Tk}, for every $U\in\mathcal B$
fix a continuous pseudonorm $|\cdot|_U$ on $H$ such that $\{x\in
H:|x|_U<1\}\subset U$. By Theorem~\ref{main}, the pseudonorm
$|\cdot|_U$ can be extended to a pseudonorm $\|\cdot\|_U$ on $G$
such that $d(G,\|\cdot\|_U)=d(H,|\cdot|_U)$. The continuity of the
identity map $(H,\tau_H)\to (H,|\cdot|_U)$ implies that
$ib(H,|\cdot|_H)\le ib(H,\tau_H)$, see \cite[3.2]{Tk}. Since the
density and the boundedness index coincide for (pseudo)metrizable
topological groups \cite[\S3]{Tk}, we conclude that
$ib(G,\|\cdot\|_U)=d(G,\|\cdot\|_U)=
d(H,|\cdot|_U)=ib(H,|\cdot|_U)\le ib(H,\tau_H)$.

Let $\tau_G$ be the smallest topology on $G$ making continuous the
quotient homomorphism $(G,\tau_G)\to (G/H,\tau_{G/H})$ and the
identity map $(G,\tau)\to (G,\|\cdot\|_U)$ for all $U\in\mathcal
B$. It is easy to see that $\tau_G$ is a Hausdorff group topology
on $G$ inducing the topology $\tau_H$ on the subgroup $H$.

Observe that the topological group $(G,\tau_G)$ can be identified
with a subgroup of the product $G/H\times\prod_{U\in\mathcal B}
(G,\|\cdot\|_U)$ of topological groups whose boundedness indices
do not exceed $ib(H,\tau_H)$ and characters do not exceed
$\log|G|$. According to \cite[3.2]{Tk}, the boundedness index of
such a product does dot exceed $ib(H,\tau_H)$ while its character
does not exceed $\chi(G/H)\cdot|\mathcal B|\le
\log|G|\cdot\chi(H,\tau_H)$. Consequently, $ib(G,\tau_G)\le
ib(H,\tau_H)$, $\chi(G,\tau_G)\le \chi(H,\tau_H)\cdot \log|G|$,
and $w(G,\tau_G)=\chi(G,\tau_G)\cdot ib(G,\tau_G)\le
\log|G|\cdot\chi(H,\tau_H)\cdot
ib(H,\tau_H)=w(H,\tau_H)\cdot\log|G|$.

\smallskip
\smallskip

\noindent{\bf Acknowledgement.} The authors express their sincere
thanks to Sasha Ravsky and Igor Protasov for valuable and
stimulating discussions (concerning stochastic matrices).


\newpage

\begin{thebibliography}{}

\bibitem[A]{A}
M.~Aigner. Combinatorial Theory. -- Springer-Verlag, 1979.

\bibitem[BRZ]{BPZ}
T.~Banakh, A.~Plichko, A.~Zagorodnyuk. Automatic continuity of
polynomial operators between topological abelian groups (in
preparation).

\bibitem[Bi]{Bi} G.~Birkhoff, Tres observaciones sobre el algebra lineal //
Rev. Univ. Nac Tucuma\'an. ser. A. -- 1946. -- V.5. --P.147-151.

\bibitem[En]{En}
R.~Engelking. {\em General Topology}. -- Warszawa: PWN, 1985.

\bibitem[F]{F} L. Fuchs. {\em Infinite Abelian Groups, I}. -- New
York: Academic Press, 1970.

\bibitem[Ga]{Ga} F.R.~Gantmacher. {\em Matrix Theory}. -- Moscow:
Nauka, 1967 (in Russian). P.556.

\bibitem[Gu]{Gu}
I.~Guran, On topological groups close to being Lindel\"of, {\em
Soviet Math. Dokl}. -- 1981. -- V.23. -- P.173--175.

\bibitem[JW]{JW}
W.~Just, M.~Weese, {\em Discovering Modern Set Theory, I} (GSM,
Vol.8. -- Providence, RI: Am. Math. Soc., 1996.

\bibitem[Ke]{Ke}
A.S.~Kechris. {\em Classical descriptive set theory}. -- Berlin:
Springer-Verlag, 1995.

\bibitem[PZ]{PZ} A.~Plichko, A.~Zagorodnyuk. Isotropic mappings
and automatic continuity of polynomial, analytic, and convex
operators // in: {\em General Topology in Banach Spaces} (T.Banakh
ed.). -- New York: Nova Publ., 2001. -- P.1--13.

\bibitem[Tk]{Tk}
M.~Tkachenko, Introduction to topological groups, {\em Topology
Appl}. -- 1998. -- V.86. -- P.179--231.

\end{thebibliography}
\end{document}